\documentclass{article}
\usepackage{arxiv}

\usepackage[capitalize,noabbrev]{cleveref}

\PassOptionsToPackage{colorlinks=true, linkcolor=blue, citecolor=blue, urlcolor=blue}{hyperref}
\usepackage{hyperref}

\usepackage{xcolor}
\usepackage{comment}
\usepackage{graphicx}
\usepackage{tcolorbox}

\title{Pseudorandomness of the Sticky Random Walk}

\author{ \href{https://orcid.org/0000-0003-2893-9469}{\includegraphics[scale=0.06]{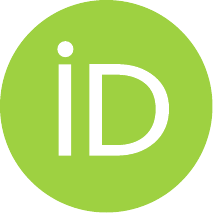}\hspace{1mm}Emile Anand} \\
	Computing and Mathematical Sciences\\
	California Institute of Technology\\
	Pasadena, CA, 91125 \\
	\texttt{eanand@caltech.edu} \\
	\And
	\href{https://orcid.org/0000-0002-6390-9401}{\includegraphics[scale=0.06]{orcid.pdf}\hspace{1mm}Chris Umans} \\
	Computing and Mathematical Sciences\\
	California Institute of Technology\\
	Pasadena, CA, 91125 \\
	\texttt{umans@cms.caltech.edu} \\
}

\date{}

\renewcommand{\shorttitle}{Pseudorandomness of the Sticky Random Walk}

\hypersetup{
pdftitle={Pseudorandomness of the Sticky Random Walk},
pdfsubject={TCS},
pdfauthor={Emile T. Anand, Chris M. Umans},
pdfkeywords={Expander Graphs · Random Walks · Sticky Random Walk · Pseudorandomness · Derandomization}
}

\begin{document}


\maketitle

\begin{abstract}
We extend the pseudorandomness of random walks on expander graphs using the sticky random walk. Building on the works of  \citep{cohen_et_al:LIPIcs.ICALP.2022.43} \citep{guruswami_et_al:LIPIcs.ITCS.2021.48}, \citep{golowich_et_al:LIPIcs.CCC.2022.27} recently showed that expander random walks can fool all symmetric functions in total variation distance (TVD) upto an $O(\lambda (\frac{p}{\min f})^{O(p)})$ error, where $\lambda$ is the second largest eigenvalue of the expander, $p$ is the size of the arbitrary alphabet used to label the vertices, and $\min f = \min_{b\in[p]} f_b$, where $f_b$ is the fraction of vertices labeled $b$ in the graph. \citep{golowich_et_al:LIPIcs.CCC.2022.27} conjectures that the dependency on the $(\frac{p}{\min f})^{O(p)}$ term is not tight. In this paper, we resolve the conjecture in the affirmative for a family of expanders. We present a generalization of \citep{guruswami_et_al:LIPIcs.ITCS.2021.48}'s sticky random walk for which \citep{golowich_et_al:LIPIcs.CCC.2022.27} predicts a TVD upper bound of $O(\lambda p^{O(p)})$ using a Fourier-analytic approach. For this family of graphs, we use a combinatorial approach involving the Krawtchouk functions used in \citep{guruswami_et_al:LIPIcs.ITCS.2021.48} to derive a strengthened TVD of $O(\lambda)$. Furthermore, we present equivalencies between the generalized sticky random walk, and, using linear-algebraic techniques, show that the generalized sticky random walk is an infinite family of expander graphs.

\end{abstract}


\pagebreak

\section{Introduction}\label{sec:Intro}

Expander graphs are undirected spectral sparsifiers of the clique with high expansion properties, and they are among the most useful combinatorial objects in theoretical computer science due to their applications in pseudorandomness and in error-correcting codes (see \cite{kowalski2013expander}\cite{expsurvey}\cite{vadhan_2013}\cite{sipser-spielman}\cite{reinfold-st-connectivity}). A $d$-regular graph $G\!=\!(V\!,\!E)$ where $|V|\!=\!n$ is said to be an $(n, d, \lambda)$-expander if (for eigenvalues $\lambda_1 \geq \lambda_2 \geq \cdots \lambda_n=-1$ of its normalized adjacency matrix $A$) the second largest magnitude of the eigenvalue of its adjacency matrix $A$ is atmost $\lambda$. In other words, $\max_{i\neq 1}|\lambda_i| \leq \lambda$. It is a well-known fact that $\lambda\leq 1$ implies $G$ is connected; so, smaller values of $\lambda$ pertain to stronger combinatorial connectivity properties (see \cite{trevisanlecturenotes}). There are several known constructions \cite{https://doi.org/10.48550/arxiv.1711.06558}\cite{ben-aroya-avraham-tashma}\cite{sipser-spielman}\cite{reingold_vadhan_wigderson_2004} of optimal (Ramanujan) expander graphs that saturate the Alon-Boppana bound \cite{10.1215/S0012-7094-03-11812-8} which characterizes ``good'' expanders. The most useful fact about expander graphs in pseudorandomness arises from the fact that random walks on them mix fast. Let $v_0, ..., v_{t-1}$ be a sequence of vertices obtained by a $t$-step walk on an expander graph $G$ with second largest eigenvalue atmost $\lambda$. Gilman's \cite{gillman} results uses the Chernoff bound to characterize the rate of mixing (this has since been improved in \cite{rao-shravas-regv} and \cite{golowich_et_al:LIPIcs.CCC.2022.27}).

\emph{Expander-Walk Chernoff Bound (Gilman)}. For graph $G = (V, E)$, let $v_0,\dots, v_{t-1}$ denote a sequence of vertices obtained from a $t$-step -walk on an expander graph $G$. For any function $f: [n]\to \{0,1\}$, let the stationary distribution of $G$ be $\pi(f) := \lim_{t\to\infty}\frac{1}{t}\sum_{i=0}^{t-1} f(v_i)$. The expander-walk Chernoff bound \cite{gillman} states that $\forall\varepsilon > 0$,
\[\Pr\left[\left|\frac{1}{t}\sum_{i=0}^{t-1}f(X_i) - \pi(f)\right| \geq \varepsilon \right] \leq 2e^{-\Omega((\lambda-\varepsilon)^2 t)}\]
For a proof of the expander walk Chernoff bound, we refer the reader to \cite{expsurvey}\cite{szemerediregularity}\cite{rao-shravas-regv}. An important direct consequence of the expander Chernoff bound is that the mixing time of a $d$-regular expander graph on $n$ vertices is atmost $O(\log n)$. 

Expander graphs have (oftentimes, surprisingly) ubiquitous applications. They were initially studied for the purpose of constructing fault-tolerant networks in \cite{impagliazzo-nisam-wigderson,chaudhari2024peer,anand2024mean,anand2024efficient}, where if a small number of channels (edges) broke down, the system could be made to be still largely intact due to its good connectivity properties if it were modeled as an expander. More recently, they have been used in representation learning theoretic settings (see \cite{deac2022expander}) to create graph neural networks that can propagate information to train models more efficiently, and in decentralized online optimization \cite{chow2016expander,lin2023online2,NEURIPS2023_a7a7180f,pmlr-v247-lin24a}. In coding theory, expander codes (created from linear bipartite expanders - see \cite{alonbipartite}) are the only known construction (see \cite{sipser-spielman}) of asymptotically good error-correcting codes which can be decoded in linear time when a constant fraction of symbols are in error.

More recent works that combine ideas from combinatorial topology and algebraic geometry have also led to the exciting study of high dimensional expanders (HDX) which are pure simplicial complexes (hypergraphs that are downwards closed under containment) where the $1$-skeletons are spectral expanders and the links exhibit good expansion properties. We direct the reader to \cite{conlon2019hypergraph} and \cite{gotlib2023high}.

One of the most important applications of expanders (which is the topic of this thesis) is on derandomization and in pseudorandomness. Suppose that there is a randomized algorithm for a language $L$ using $n$ bits such that: If a string $x\in L$, then the algorithm accepts with probability $1$. If a string $x\notin L$, then the algorithm rejects with probability atleast $1/2$. Our goal is to reduce the error probability of the algorithm. If we repeat the algorithm $t$ times then the error probability goes down to $1/2^t$, which is ideal. However, the number of random bits used by the algorithm is then equal to $nt$, which is very large. One work around is to ``reuse the randomness by weakening our independent choices to correlated choices on an expander graph'' \cite{guruswami_berkeley_lecture_notes}. If we start at a random vertex in $G$ (which is a random number in $\{0,\dots,n\}$ which uses $\log n$ random bits) and pick random neighbors of $v$, then since a good expander has degree $d=O(1)$ and since we need $\log d = O(1)$ bits, we can continue this process till we pick $t$ vertices overall, and the overall number of random bits that we would need would be equal to $\log n+O(t)$. Further, by the expander mixing lemma, for $t\gg O(\log n)$ the sequence of vertices will still be extremely close to uniformly random. 

This makes expander graphs invaluable in the field of pseudorandomness. Consider the $t$-step expander random walk which generates a sequence of vertices $v_0, \dots, v_{t-1}$. We are then interested in the degree to which $(v_0, \dots, v_{t-1})$ ``fool'' classes of test functions, where we say that a test-function $T$ is $\epsilon$-fooled by a pseudorandom function $g:X\to[\chi]$ if the statistical distance between distributions $T(g(X))$ and $T(U)$ (here $U$ is the uniform distribution on $[\chi]$), is less than $\epsilon$.

Ta-Shma's breakthrough construction of optimal $\epsilon$-balanced codes \cite{tashma} which showed that expander random walk can fool the extremely sensitive parity function led to an exciting series of results that showed the pseudorandomness of expander random walks for increasingly many classes of test functions. \cite{guruswami_et_al:LIPIcs.ITCS.2021.48} introduces the `sticky random walk', a canonical two-vertex expander which can be thought of as a Markov chain on two states, where the probability of switching between states is $\frac{1+\lambda}{2}$ and the probability of staying at the same state is $\frac{1-\lambda}{2}$. \cite{guruswami_et_al:LIPIcs.ITCS.2021.48} goes on to use the discrete orthogonal Krawtchouk functions to show that the Hamming weight distribution of the sticky random walk can fool any symmetric function with error $O(\lambda)$, where $\lambda < 0.16$. \cite{cohen-peri-tashma}\cite{cohen_et_al:LIPIcs.ICALP.2022.43} then used Fourier analysis to expand this result. Specifically, they showed that test functions computed by $\mathrm{AC}^0$ circuits and symmetric functions are fooled by the random walks on the full expander random walk, but only for balanced binary labelings. These works culminate in \cite{golowich_et_al:LIPIcs.CCC.2022.27,anand2025towards} which establishes that random walks on expanders where the vertices are labeled from an arbitrary alphabet can fool symmetric functions (upto an $O(\lambda)$ error), permutation branching programs, and low-depth $\mathrm{ACC}^0$-circuits. Specifically, we are interested in the result concerning symmetric functions (Theorem 2) which we restate below.

\emph{Fooling symmetric functions (Corollary 4 of \cite{golowich_et_al:LIPIcs.CCC.2022.27})}. For all integers $t \geq 1$ and $p \geq 2$, let $G =
(G_i)_{1\leq i\leq t-1}$ be a sequence of $\lambda$-spectral expanders on a shared vertex set $V$ with labeling $\mathrm{val} : V \to [p]$
that assigns each label $b \in [p]$ to $f_b$-fraction of the vertices. Then, for any label $b$, we have that the total variation distance between the number of $b$'s seen in the expander random walk and the uniform distribution on $[p]$ has the following bound (where $[\Sigma(Z)_b]$ counts the number of occurrences of $b$ in $Z$):
\[\mathrm{TVD}([\Sigma(\mathrm{RW}^t_\mathcal{G})]_b, [\Sigma(U[d])]_b) \leq O\left(\left(\frac{p}{\min_{b\in [p]} f_b}\right)^{O(p)}\cdot \lambda\right)\]

\cite{golowich_et_al:LIPIcs.CCC.2022.27} asks whether the $(\frac{p}{\min_{b\in [p]} f_b})^{O(p)}$ dependence in the upper bound of the total variation distance is tight. In this paper, we answer in the negative for a family of graphs. Specifically, we present a family of generalized sticky random walks (where the alphabet size can be arbitrarily large), where in \Cref{theorem 5.3} we find that the optimal TVD is $O(\lambda)$, for $\lambda < 0.27$, whereas Corollary 4 in \cite{golowich_et_al:LIPIcs.CCC.2022.27} predicts a bound of $O(\lambda p^{2p})$, which provides evidence that the $(\frac{p}{\min_{b\in[p] f_b}})^{O(p)}$ factor is not tight. \cite{guruswami_et_al:LIPIcs.ITCS.2021.48} studied the sticky random walk because it was an ``essential step'' to understanding the full expander random walk - specifically, theorem $4$ in \cite{guruswami_et_al:LIPIcs.ITCS.2021.48} shows that every $\lambda$-parameterized sticky random walk is bijective with a corresponding expander graph. We extend their result in \Cref{theorem 7.1} by showing that our generalized sticky random walk (parameterized by $\lambda$ and $p$ also correspond to expander graphs with a linearly-reduced spectral expansion of $\lambda p$. We then show that our generalized sticky random walk reduces from \cite{guruswami_et_al:LIPIcs.ITCS.2021.48}'s two-vertex sticky random walk in \Cref{theorem 7.2}. Finally, in \Cref{appendix 2} we provide a novel alternate treatment of the Krawtchouk functions into the complex domain which can be used to attain an $O(\lambda p^{p})$ bound on the TVD.\\

\section{Preliminaries: Notation and Conventions}\label{sec:Preliminaries}
This section describes the basic notation and problem setup that is used throughout the paper.\\

For any $n\in\NN$, let $[n]=\{1,...,n\}$ and $\ZZ_{n}=\{0, ..., n-1\}$. Next, let $[n]^k$ denote $k$ copies of elements in $[n]$, and let $\ZZ_n^k$ denote $k$ copies of elements in $\ZZ_n$. Furthermore, let $\genfrac() {0pt}{1}{[n]}{k}$ denote the set of all $k$-sized subsets of $[n]$, which has cardinality $\genfrac() {0pt}{1}{n}{k}$. For any $n$-bit string $s$, let $|s|$ denote the Hamming weight (the Hamming distance from $0^n$) of $s$. Similarly, let $|s|_i$ denote the number of $i$'s in $s$. We generalize the notion of counting the number of occurrences of any character $\chi\in \ZZ_{p}$ for $p\geq 2$ with the symmetric function $\Sigma(x):\ZZ_{p}^n \rightarrow \ZZ_n^p$, where $\Sigma(x)$ is a vector that counts the number of occurrences of each $\chi\in \ZZ_{p}$. Specifically, for all $\chi\in \ZZ_{p}$ and for all $x\in \ZZ_{p}^n$, we can write that $[\Sigma(x)]_\chi= |\{i\in x: x_i=\chi\}|=|x|_\chi$. \\

Let $\mathrm{Ber}(q)$ denote the Bernoulli distribution on $\{0,1\}$, such that if $X\sim \mathrm {Ber}(q)$, then $\Pr[X=1]=q$ and $\Pr[X=0]=1-q$. Next, let $\mathrm{Bin}(n, 1/2)$ denote the binomial distribution of $\sum_{i=1}^n b_i$ with independent choices of $b_i \sim$ Ber($1/2$). Let $\mathrm U_p^n = \mathrm U[\{0,\dots,p-1\}]^n$  denote $n$ samples of the uniform distribution on $\ZZ_p$, where each number is sampled with probability $1/p$. Then, $[\Sigma(\mathrm U_p^n)]_0$ reports the number of $0$'s in an $n$-bit sample from the uniform distribution on $\ZZ_p$. Furthermore, we write that $x\in A$ if $x$ is an element of $A$, and $x \in_U A$ if $x$ is an element chosen uniformly randomly from $A$. Finally, we use $\overline{C}$ to denote the complement of a set $C \subseteq \Omega$, and for any two sets $A,B\subseteq\Omega$, we define their symmetric difference $A \Delta B$ as $(A\cap \overline{B})\cup (B\cap\overline{A})$.\\

\begin{definition}[$(n,d,\lambda)$-expanders]
\emph{A $d$-regular graph $G=(V,E)$ where $|V|=n, |E|=m$ is said to be an $(n, d, \lambda)$-expander if (for eigenvalues $1=\lambda_1 \geq \lambda_2 \geq \cdots \lambda_n=-1$ of its degree-normalized adjacency matrix $A$) the second largest magnitude of the eigenvalue of its adjacency matrix $A$ is atmost $\lambda$. In other words, $\max_{i\neq 1}|\lambda_i| \leq \lambda$. Intuitively, the spectrum ($\lambda_1, ..., \lambda_n$) of an expander graph approximates the spectrum of the complete graph, which makes expanders a spectral sparsification of the clique.}\\
\end{definition}

\begin{definition}[Sticky Random Walk]
    \emph{The Sticky Random Walk (SRW) $S(n,\lambda)$ is a distribution on $n$-bit strings that represent $n$-step walks on a Markov chain with states $\{0,1\}$ such that for each $s\sim S(n,\lambda), \Pr[s_{i+1}=b|s_i=b] = \frac{1+\lambda}{2}$, for $b\in\{0,1\}$, and $s_1\sim $ Ber($1/2$) such that $\Pr[s_1=0]=\Pr[s_1=1]=1/2$. As $\lambda \to 0$, the distribution of strings from the Markov chain converges to the distribution of $n$ independent coin-flips.}
\end{definition}

\begin{center}
\includegraphics[scale=0.25]{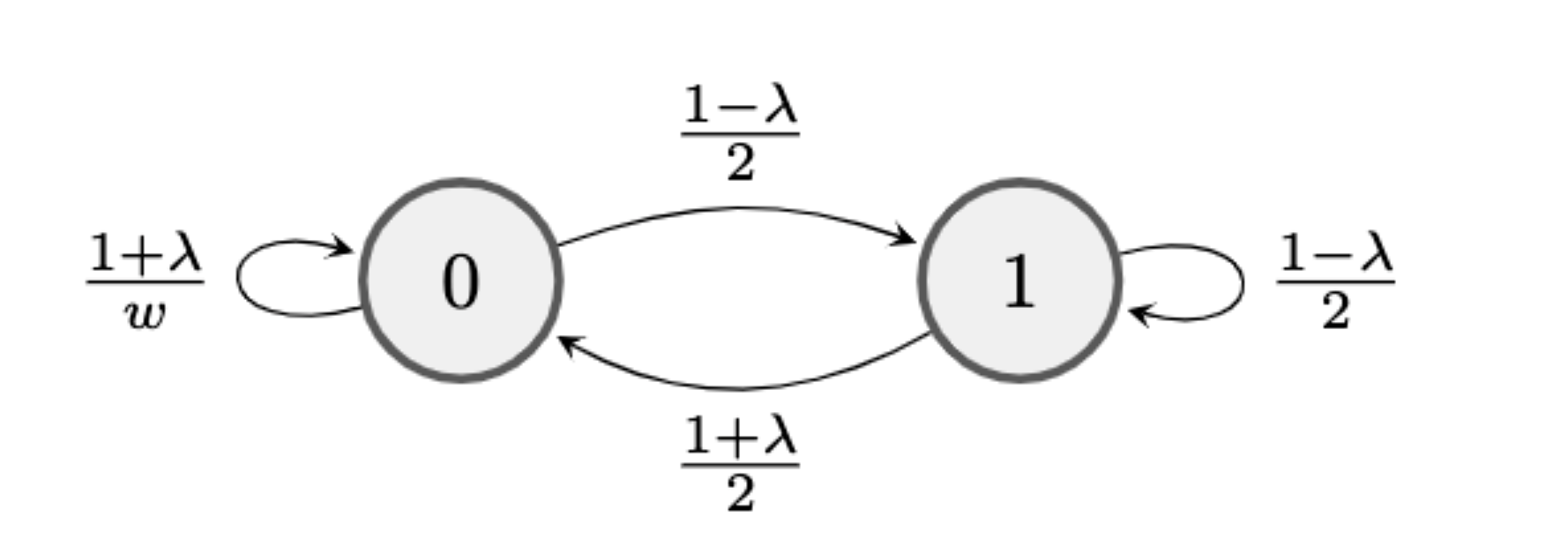}

\textbf{Figure 1:} The Markov chain of the sticky random walk $S(n,\lambda)$.
\end{center}

\begin{definition}[Krawtchouk Functions]
    \emph{The Krawtchouk function $K_k:\ZZ_{n+1} \rightarrow \RR$ is defined, for $\ell\in \ZZ_{n+1}$ and an arbitrary $n$-bit string $\alpha$ where $|\alpha| = \ell$, ($|\cdot|$ is the number of $0$'s) such that:} \[K_k(\ell) := 
  \sum_{\substack{y\in\{0,1\}^n \\ |y|=k}} (-1)^{\alpha \cdot y} = \sum_{t=0}^k (-1)^t \genfrac(){0pt}{0}{\ell}{t}\genfrac(){0pt}{0}{n-\ell}{k-j}\]
\end{definition}

\begin{lemma}[Krawtchouk orthogonality]
\emph{The Krawtchouk function is orthogonal with respect to functions mapping $[n]_0 \to \RR$ with respect to the inner product. The proof of this is deferred to \Cref{technical_lemmas}. Specifically}:
\[\langle K_r, K_s \rangle = \EE_{b\sim\mathrm{Bin}(n,1/2)}[K_r(b)K_s(b)] = \begin{cases}
0, & \text{if } r\neq s \\
\genfrac(){0pt}{4}{n}{s}, & \text{if } r=s\end{cases}\]
\end{lemma}

\begin{lemma}[Krawtchouk invariance]
\emph{The Krawtchouk function is invariant against choices of $\alpha$ which satisfy $|\alpha|=\ell$. This statement is proven in \cite{Samorodnitsky98onthe}. We formally state this property below:}
\[\EE_{\substack{|L|=k \\ \mathrm{Fixed}\text{ }A \in \ZZ_p^n \\ |A|=\ell} }[(-1)^{A\cdot B}] = \EE_{\substack{|L|=k \\ \mathrm{Random}\text{ }A' \in \ZZ_p^n \\ |A'|=\ell} }[(-1)^{A'\cdot B}]\]
\end{lemma}

\begin{definition}[Total variation distance] 
    \emph{Given a measure space $(\Omega, \mathcal{F}, \mu)$ and a $\sigma$-algebra $\mathcal{A}\subseteq \mathcal{F}$, the total variational distance $d_{\text{TV}}(\mu_1, \mu_2)$ between probability measures $\mu_1, \mu_2:\mathcal{F}\rightarrow\RR$ is \[d_{\text{TV}}(\mu_1, \mu_2) = \sup_{A\in\mathcal{A}}|\mu_1(A)-\mu_2(A)|\] Similarly, the $\ell_1$ distance between $\mu_1$ and $\mu_2$ is defined as $d_{\ell_1}(\mu_1,\mu_2)=2d_\text{TV}(\mu_1,\mu_2)$. Additionally, for countable $\Omega$ and $\mathcal{A}=2^\Omega$, we have $d_{\ell_1}(\mu_1, \mu_2) = \sum_{x\in\Omega}|\mu_1(x)-\mu_2(x)|$. Transitively,
\[d_\text{TV}(\mu_1, \mu_2) = \frac{1}{2}\sum\limits_{x\in\Omega}|\mu_1(x)-\mu_2(x)|\]}
\end{definition}

\begin{definition} \label{definition: 5 (primitive root of unity)} [Primitive root of unity]
    \emph{For $p\geq 2$, let $\omega_p$ denote the $p^{th}$ primitive root of unity if it satisfies $(\omega_p)^p=1$, and if there does not exist $q\in\NN$ where $q<p$ such that $(\omega_p)^q=1$. Specifically, the multiplicative order of the $p^{th}$ primitive root of unity must be $p$. Then, for $1\leq k<p$, we must have that
\[\sum\limits_{j=0}^{p-1} (\omega_p)^{k j} = \omega_p^{k\cdot 0} + \omega_p^{k\cdot 1} + ... + \omega_p^{k\cdot(p-1)} =0\]}
\end{definition}

\section{A Generalization of the Sticky Random Walk}\label{sec:generalizing}

We consider the case where the vertices of the sticky random walk (SRW) can be labeled with an arbitrary alphabet $\ZZ_p$, since a decomposition of the vertex-set $V = \{0,\dots,p-1\} = V_1 \sqcup ... \sqcup V_q$ for $q\geq 2$, could allow us to model random walks where the probability of transitioning between different states is asymmetric, while allowing us to study the pseudorandomness of random walks on graphs with vertices with more complex labelings which has already been explored in \cite{golowich_et_al:LIPIcs.CCC.2022.27}. This section generalizes the sticky random walk on $p$ characters, and provides the context for bounding the total variation distance between the sticky random walk and $\mathrm U_p^n$.\\

\begin{definition}[The generalized sticky random walk] \emph{The generalized sticky random walk $S(n, p, \lambda)$ is an $n$-step long, $p$-symbol walk on a Markov chain with $p$ states $\ZZ_p$ labeled as vertices on a complete graph with self-loops $J_p$, where $s_0 \in_U \ZZ_p$ and at each subsequent step, we either stick to the same state with probability $\frac{1}{p} + (p - 1)\lambda$, or change to any other state with uniform probability $\frac{1}{p} - \lambda$. So, for instance, comparing $S(n, 4, \lambda)$ to $S(n, 4, 0) = U_4^n$ (the uniform random walk on 4 vertices) yields the following Markov chain graphs:}

\begin{center}    
    \includegraphics[scale=0.18]{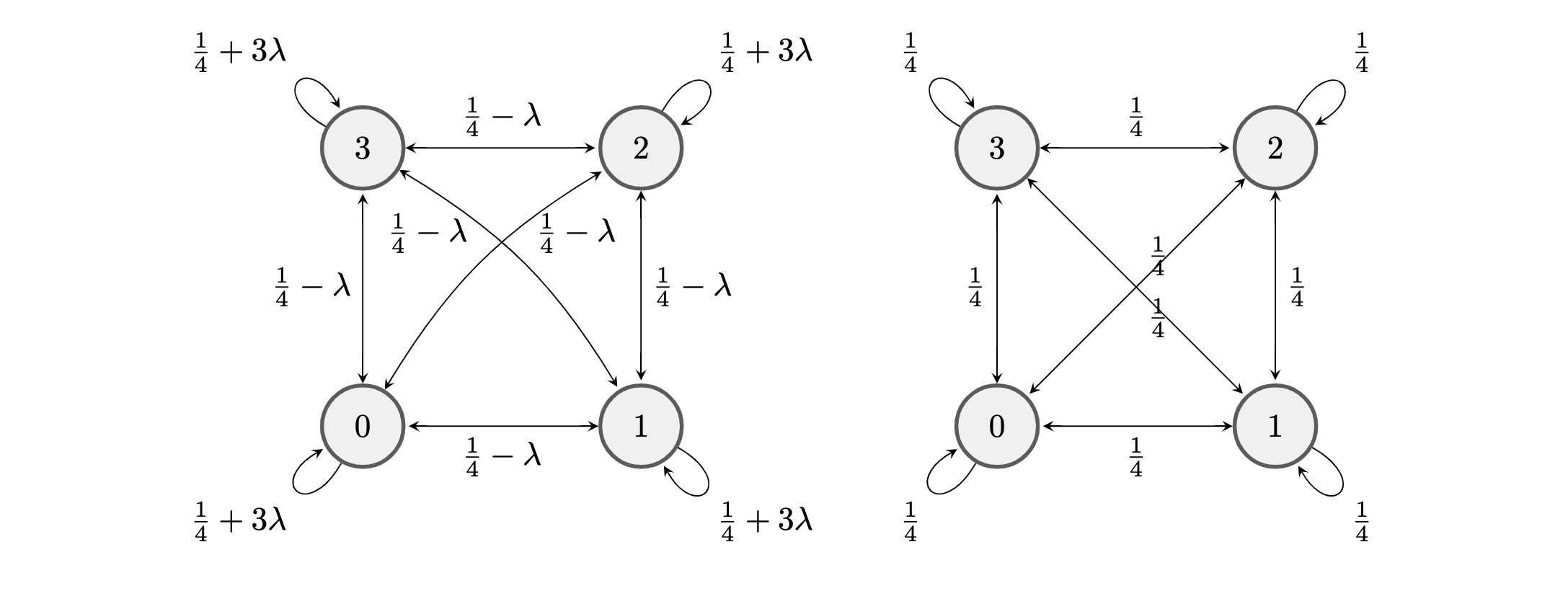}
    \textbf{\newline Figure 2: } Markov chains of the generalized sticky random walk on 4 states and the uniform distribution on $4$ states. The figure on the left corresponds to $S(n,4,\lambda)$, the $\lambda$-biased sticky random walk on four vertices, and the figure on the right corresponds $S(n,4,0)=U_4^n$, the unbiased random walk on four vertices.
\end{center}
\end{definition}
\vspace{0.1cm}

\begin{proposition}[Probability Invariance Under Permutations]
For any $x_1, x_2 \in \ZZ_p^n$, we have $\text{Pr}[x_1]=\text{Pr}[x_2]$ iff $|({i, i + 1})|$ such that $(x_1)_i = (x_1)_{i+1}$ is equal to $|({j, j + 1})|$ such that $(x_2)_j = (x_2)_{j+1}$. This can be shown by considering the products of conditional probabilities on each state. A slight weakening of this statement is that for any permutation $\pi$ such that $\pi:\ZZ_p \rightarrow \ZZ_p$, we must have that $\text{Pr}[x_1 ... x_n] = \text{Pr}[\pi(x_1)...\pi(x_n)]$. Consider the case of $p=2$. Then, the lemma yields that $\text{Pr}(x) = \text{Pr}(\bar{x})$, which says that inverting the labels of a string from the sticky random walk does not change its probability. This proposition extends the same argument to all $p \in \NN$ for $p\geq 2$.\\
\end{proposition}

\begin{proposition}[Krawtchouk Orthogonality]\label{prop: 2}
\emph{The orthogonality of the Krawtchouk function $K_k(\ell)$ implies that for any function $f:\ZZ_{n+1}\rightarrow\RR$, there exists a unique expansion $f(\ell) = \sum\limits_{k=0}^n \hat{f}(k) K_k(\ell)$, where for $0\leq k\leq n$,} \[\hat{f}(k) = \EE
 \bigg[{{n}\choose {k}}f(b)K_k(b)\bigg]\]
\end{proposition}

\begin{definition}[Probability ratio]
    \emph{Let $q:\ZZ_{n+1}\rightarrow\RR$, where $q(\ell)=\frac{\Pr\limits_{s\sim S}[|s|_0=\ell]}{{n\choose \ell}(p-1)^{n-\ell}}p^n$. Intuitively, $q(\ell)$ is the ratio of the probability of getting a string with $\ell$ 0s from the generalized sticky random walk $S(n,\lambda,p)$ to the probability of getting a string with $\ell$ 0s from $\mathrm U_p^n$}.\\
\end{definition}

\begin{lemma}\label{lemma: 3.1}[Krawtchouk coefficient of the probability ratio]
\emph{Expanding $q(\ell)$ through the Krawtchouk function expansion in Proposition 3.2 yields that:}
\[\hat{q}(k) = \frac{1}{{n\choose k}(p-1)^{n-k}}\EE\limits_{s\sim S(n,p,\lambda)}[K_k(|s|_0)]\]
\end{lemma}
\begin{proof} Writing the expected value of $K_k(|s|_0)$ using $q(b)$ and $K_k(b)$, we get: \begin{align*}
        \hat{q}(k) &= \frac{1}{{n\choose k}(p-1)^{n-k}}\sum\limits_{b=0}^n {n\choose b} \frac{(p-1)^{n-b}}{p^n} q(b) K_k(b) \\
        &= \frac{1}{{n\choose k}(p-1)^{n-k}}\sum\limits_{b=0}^n \Pr\limits_{s\sim S(n,p,\lambda)}[|s|_0 = b]K_k(b) \quad \text{(substituting $q(b)$)}\\
        &= \frac{1}{{n\choose k}(p-1)^{n-k}} \EE\limits_{s\sim S(n,p,\lambda)}[K_k(|s|_0)] \quad \text{(by definition of $\EE\limits_{s\in S(n,p,\lambda)}[K_k(|s|_0)]$)}\qedhere
\end{align*}\end{proof}

\begin{lemma}
\emph{For $s\in S(n,p,\lambda)$, we have that $\Pr[|s|_0=\ell]=\frac{1}{p^n} \sum\limits_{k=0}^n K_\ell(k) \EE\limits_{s\sim S(n,p,\lambda)}[K_k(|s|_0)].$}
\end{lemma}
\begin{proof} Writing out $\Pr[|s|_0=\ell]$ in terms of the probability ratio $q(\ell)$, we get:
\begin{align*}
        \Pr[|s|_0=\ell] &= \frac{{n\choose \ell}(p-1)^{n-\ell}}{p^n} q(\ell) \\
        &= \frac{{n\choose \ell}(p-1)^{n-\ell}}{p^n} \sum\limits_{k=0}^n \hat{q}(k) K_k(\ell) \quad \text{(Krawtchouk expansion of $q(\ell)$)}\\
        &= \frac{{n\choose \ell}(p-1)^{n-\ell}}{p^n} \sum\limits_{k=0}^n \frac{K_k(\ell)}{{n\choose k}(p-1)^{n-k}} \EE[K_k(|s|_0)] \quad (\text{Lemma 3.1})\\
        &= \frac{1}{p^n} \sum\limits_{k=0}^n \frac{{n\choose \ell}}{{n\choose k}} \frac{(p-1)^{n-\ell}}{(p-1)^{n-k}} \EE[K_k(|s|_0)] K_k(\ell)\\
        &= \frac{1}{p^n} \sum\limits_{k=0}^n K_\ell(k)  \EE\limits_{s\sim S(n,p,\lambda)}[K_k(|s|_0)] \quad \text{(By the reciprocity relation)} \qedhere
    \end{align*}
\end{proof}

Therefore, we observe that to compute $\Pr[|s|_0 = \ell]$, it is imperative to calculate the expected value of the Krawtchouk function. \Cref{sec:expectations} is devoted to computing $\EE_{s\sim S(n,p,\lambda)}[K_k(|s|_0)]$.

\section{Expectation of the Krawtchouk Function for the p-vertex Sticky Random Walk}\label{sec:expectations}

\begin{definition}[Shift Function]\label{definition:shift}\emph{Given any set $T \subseteq[n]$ such that $|T|=k$, let $a_1 < ... < a_k$ be the elements of $T$ in increasing order. Then, for any $c\in \ZZ_p$, let }
\[\mathrm{shift}_c(T) = \sum\limits_{i=0}^{\floor{|k-c|/p}}(a_{c+i p}-a_{c+i p-1})\]
\emph{Then, for any $c$ such that $k \mod p = -c$, and for any $d\in \ZZ_{n+1}$, let $\phi_c(d)$ denote the number of subsets of $[n]$ of size $k$ such that $\mathrm{shift}_c(T)=d$. Note that for any $t\leq 0, a_t=0$}.\\
\end{definition}

\begin{lemma} \emph{The expected value of the Krawtchouk function is given by: 
\[\EE[K_k(|s|_0)]=\begin{cases}(p-1)^{n-k} \sum\limits_{d=k}^{n-k}\phi_0(d)\lambda^d, & \text{if }c= k\!\!\!\!\mod p \equiv 0 \\ 0, & \text{ if }c= k\!\!\!\!\mod p \not\equiv 0\end{cases}\]}
\end{lemma}
\begin{proof} By the formula of expected values, we have that:
\begin{align*}
\EE[K_k(|s|_0)] &= \sum\limits_{s\sim S(n,p,\lambda)} \Pr[s] \sum_{\substack{y\in \ZZ_2^n \\ |y|_0 = k}} (-1)^{y\cdot s} \\
&= \sum\limits_{s\sim S(n,p,\lambda)} \Pr[s] \sum_{\substack{y\in \ZZ_2^n \\ |y|_0 = k}} (-1)^{\sum\limits_{i=1}^n y_i\cdot s_i} \\
&= \sum\limits_{s\sim S(n,p,\lambda)} \Pr[s] \sum_{\substack{y\in \ZZ_2^n \\ |y|_0 = k}} \prod\limits_{i=1}^n (-1)^{y_i\cdot s_i}
\end{align*}
We note then that the dot-product on the exponent of $(-1)$ only takes the summation of the element-wise product of $\alpha$ and $y$ for positions on $y$ that are strictly non-zero. Therefore, we can rewrite the summation by considering the indices corresponding to locations of non-zeros in $y$, and instead take the summation of the dot-product along these indices. So, for $T=\{a_1 < ... < a_{n-k}\}$, we have that:
\[\EE[K_k(|s|_0)]=\sum\limits_{s\sim S(n,p,\lambda)} \Pr[s] \sum\limits_{T\in{{[n]}\choose {n-k}}}  \prod\limits_{i\in T}(-1)^{s_i}\]
Further, choosing a $T\in {{[n]}\choose {n-k}}$ implies a choice of $\overline{T} = {{[n]}\choose k}=[n]\setminus T$. Hence, the summation reduces to:\\
\begin{align*}
    \EE[K_k(|s|_0)] &= \sum\limits_{s\sim S(n,p,\lambda)} \Pr[s] \sum\limits_{T\in{{[n]}\choose k}}  \prod\limits_{i\in \overline{T}}(-1)^{ s_i} \\
    &= \sum\limits_{T\in{{[n]}\choose k}} \EE\limits_{s\sim S(n,p,\lambda)} \bigg[ \prod \limits_{i\in \overline{T}}(-1)^{ s_i}\bigg] \quad \text{(definition of expectations)}
\end{align*}
Next, observe that the sticky random walk is a Markov chain where $(-1)^{s_i}=(-1)^{s_{i-1}}$ with probability $1/p +(p-1)\lambda)$. So, we can instead model the transitions of strings from the sticky random walk as random variables $u$, where $u_1$ is uniformly distributed in $\ZZ_p$ and for $i\geq 2$, $u_i$ is uniformly distributed on $(1-\lambda)U[\ZZ_p]+\lambda \cdot \mathbbm{1}_0$. To provide an intuition for this refactorization, $(1-\lambda)U[\ZZ_p]$ is the 'base' probability of switching to any vertex and $\lambda$ is the additional probability of staying on the same vertex. 

Then, since $s_i = \sum_{i\in T} \sum_{j=1}^i u_j$, we write that:
\begin{align*}
        \EE\limits_{s\in S(n,p,\lambda)}\bigg[\prod\limits_{i\in \overline{T}} (-1)^{ s_i}\bigg] &= 
        \EE\limits_{ s\in S(n,p,\lambda)}\bigg[(-1)^{\sum\limits_{i\in \overline{T}} \sum\limits_{j=1}^i u_j}\bigg] \\
        &= \prod\limits_{j=1}^{a_{n-k}} \EE\limits_{s\in S(n,p,\lambda)}\bigg[(-1)^{\sum\limits_{i\in \overline{T}; i\geq j} u_j}\bigg] \quad \text{(independence of $u_j$'s)}
    \end{align*}

When $j=1$, we get that:
\begin{align*}\EE\bigg[(-1)^{\sum\limits_{i\in \overline{T}; i\geq 1} u_1}\bigg] &= \EE[(-1)^{|\overline{T}|u_1}] \\
&= \begin{cases}
1, &\text{if }|\overline{T}| \mod p \equiv 0 \\
0, &\text{otherwise}\end{cases}
\end{align*}
Conversely, when $j\geq 2$, let $T_j = \{i\in \overline{T}; i\geq j\}$. Then, 
\begin{align*}\EE\bigg[(-1)^{\sum\limits_{i\in \overline{T}; i \geq j}u_j}\bigg] &= \EE[(-1)^{|T_j|u_j}] \\
&= \begin{cases}
1, & \text{if } |T_j| \mod p \equiv 0 \\
\EE[(-1)^{u_j}], & \text{ otherwise}\\
\end{cases}
\end{align*}

Next, observe that for $j\geq 2$, $\EE[(-1)^{u_j}]=\lambda$.
\begin{proof} To show this, we write the expression for $u_j, j\geq 2$ in the exponent and take the expected value of $(-1)^{u_j}$.
\begin{align*}
\EE[(-1)^{u_j}] &= \EE[(-1)^{(1-\lambda)U[\mathbb{Z}_p]+\lambda\cdot\mathbbm{1}_0}] \quad \text{(since $u_j \sim (1-\lambda)U[\mathbb{Z}_p]+\lambda \cdot \mathbbm{1}_0$)}
\\
&= \sum\limits_{k=0}^{p-1} (-1)^{k} \Pr[u_j=k] \quad \text{(definition of expectation)}\\
&= (-1)^0 \Pr[u_j=0] + (-1)^1 \Pr[u_j=1] + ... + (-1)^{p-1} \Pr[u_j=p-1] \\
&= \bigg(\frac{1}{p} + \lambda\bigg(\frac{p-1}{p}\bigg)\bigg) - \bigg(\frac{1}{p} - \frac{\lambda}{p}\bigg) + ... + (-1)^{p-1} \bigg(\frac{1}{p} - \frac{\lambda}{p}\bigg)\\
&= \frac{1}{p}\sum\limits_{k=0}^{p-1}(-1)^k - \frac{\lambda}{p}\sum\limits_{k=0}^{p-1}(-1)^k + \lambda (-1)^0 \\
&= \lambda \qedhere
\end{align*}\end{proof}
Then, for $k \mod p \equiv 0$, we have that:
\begin{equation*}\begin{split}
    \EE[K_k(|s|_0)] &= \sum\limits_{T\in{[n]\choose k}}\EE\limits_{\substack{s\sim S(n,p,\lambda)}}\bigg[\prod\limits_{i\in \overline{T}}(-1)^{s_i}\bigg] \\
    &=\sum\limits_{T\in{[n]\choose k}} \prod\limits_{i\in \overline{T}} \EE\limits_{\substack{s\sim S(n,p,\lambda)}}[(-1)^{s_i}] \quad \text{(independence of $(-1)^{s_i}$)}\\
    &= \sum\limits_{T\in{[n]\choose k}}\prod\limits_{j=1}^{a_{n-k}}\lambda \quad \text{(since $|T|=a_{n-k}$)}\\
    &= \sum\limits_{T\in{[n]\choose k}}\lambda^{a_{n-k}}
\end{split}
\end{equation*}
We then parameterize the summation over every possible value of the shift of T (for $k\!\!\mod p \equiv 0$), where the shift function is given in \Cref{definition:shift}.
\begin{align*}
\EE[K_k(|s|_0)] &=  \sum\limits_{d=k}^{n-k} \bigg(\sum_{\substack{T\in{[n]\choose k}\\ \text{shift}_0\text{(T)=d}}} 1\bigg) \lambda^d \\
&=  \sum\limits_{d=k}^{n-k} \phi_0(d) \lambda^d
\end{align*}
This yields the claim.
\end{proof}

\vspace{0.2cm}

\begin{lemma}\emph{For $c\in\NN$ where $0\leq c\leq p$, and for $d\in\NN$ where $k\leq d\leq n-k$, the number of $k$-sized subsets of $[n]$ that satisfy $\text{shift}_c(T)=d$} is:
\[\phi_c(d) = \frac{1}{p^k} \sum_{\substack{T\in{{[n]}\choose k}\\\text{shift}_c(T)=d}}1 = \frac{1}{p^k}{{d-1}\choose{\floor{\frac{|k-c|}{p-1}}-1}}{{n-d}\choose {\floor{\frac{|k-c|}{p-1}}}}\]
\end{lemma}
\begin{proof}
To determine $\phi_c(d)$, we count the total number of ways to choose $a_1\!\!<\!\!a_2\!\!<\!\!...\!\!<\!\!a_k$ such that the lengths of the intervals $(a_c-a_{c-1})+(a_{c+p}-a_{c+p-1})+(a_{c+2p}-a_{c+2p-1})+...=d$, where for any $j\leq 0$, $a_j=0$. To do this, we combine each element-wise interval $(a_{c+ip}, a_{c+ip-1})$ to form a contiguous interval of length $d$ (starting from $a_{c-1}=0$). The remaining contiguous region that excludes these intervals must then have a length of $n-d$. We then abstract the number of ways to count $a_1<...<a_k$ by counting the number of intervals that have a length of $d$ when combined, such that the remaining intervals have a length $n-d$. From a length of $d-1$ (accounting for $a_0=0$), we need to select intervals that form a length of $\floor{|k-c|/(p-1)}-1$ since they represent the number of choices of elements of $T$ that are index-separated by $p$. Similarly, from a length of $n-d$, we need to select intervals that form a length of $\floor{|k-c|/(p-1)}$ possible intervals, since they represent every other element of $T$. This second constraint is to ensure that the total length of the intervals chosen is exactly $n$. Finally, we divide by the maximum number of repetitions to prevent duplicates, which is $p^k$. Hence, we write that:

\[\sum_{\substack{ T\in{{[n]}\choose k}\\\text{shift}_c(T)=d}}1 = {{d-1}\choose{\floor{\frac{|k-c|}{p-1}}-1}}{{n-d}\choose {\floor{\frac{|k-c|}{p-1}}}}\]
Therefore,
\[\phi_c(d) = \frac{1}{p^k} \sum_{\substack{ T\in{{[n]}\choose k}\\\text{shift}_c(T)=d}}1 = \frac{1}{p^k} {{d-1}\choose{\floor{\frac{|k-c|}{p-1}}-1}}{{n-d}\choose {\floor{\frac{|k-c|}{p-1}}}} \qedhere\]
\end{proof}
\begin{corollary}
\emph{By combining the results from Lemmas 5.1 and 5.2, we have that the expectation of the Krawtchouk function is: \[\EE[K_k(|s|_0)]=\begin{cases} \frac{1}{p^k}\sum\limits_{d=k}^{n-k} {{d-1}\choose {\floor{\frac{k}{p-1}}}-1}{{n-d}\choose{\floor{\frac{k}{p-1}}}} \lambda^d, & \text{  if } k\!\!\!\!\mod p \equiv 0 \\ 0, & \text{   if } k\!\!\!\!\mod p \not\equiv 0\end{cases}\]}
\end{corollary}

Thus, having computed $\EE[K_k(|s|_0)]$, we are now prepared to upper-bound the total variation distance between $[\Sigma(S(n,p,\lambda))]_0$ and $[\Sigma(\mathrm U_p^n)]_0$. We devote \Cref{sec:upperbounds} to deriving an optimal upper bound of $O(\lambda)$.\\

\section{Upper Bounds for the Total Variation Distance}\label{sec:upperbounds}

\begin{lemma}\emph{The total variational distance between the generalized sticky random walk on $p$ vertices and the uniform distribution on $p$ states is given by:} \[\mathrm{TVD}([\Sigma(S(n,p,\lambda))]_0, [\Sigma(\mathrm U_p^n)]_0) = \frac{1}{2}\EE\limits_{b\sim \mathrm U_p^n}[|q(b)-1|]\]
\end{lemma}

\begin{proof} We write the expression of the total variation distance between the $n$-step sticky random walk on $p$ states and $n$-samples from the uniform distribution on $p$ states. This then yields:
\begin{align*}
        \TVD([\Sigma(S(n,p,\lambda))]_0, [\Sigma(\mathrm U_p^n)]_0) &= \frac{1}{2}\sum\limits_{\ell=0}^n \bigg|\Pr[|s|_0=\ell] - \frac{{n\choose\ell}(p-1)^{n-\ell}}{p^n}\bigg| 
        \\
        &=  \frac{1}{2}\sum\limits_{\ell=0}^n \bigg|{n\choose\ell}q(\ell)\frac{(p-1)^{n-\ell}}{p^n} - \frac{{n\choose\ell}(p-1)^{n-\ell}}{p^n}\bigg| \\
        &=  \frac{1}{2}\sum\limits_{\ell=0}^n \left|\frac{{n\choose\ell}}{p^n} (p-1)^{n-\ell}(q(\ell)-1)\right| \\
        &=  \frac{1}{2} \sum\limits_{\ell=0}^n \left|\Pr[\ell](q(\ell)-1)\right| 
        \\
        &=  \frac{1}{2} \EE\limits_{b\sim [\Sigma(\mathrm U_p^n)]_0}[|q(b)-1|] \qedhere
    \end{align*}
\end{proof}

\vspace{0.5cm}

\begin{corollary}\emph{The total variational distance between the generalized sticky random walk and the uniform distribution on $p$ states has the following upper bound as a result of convexity:}
\[\mathrm{TVD}([\Sigma(S(n,p,\lambda))]_0, [\Sigma(\mathrm U_p^n)]_0) \leq \frac{1}{2} \sqrt{\EE\limits_{b\sim [\Sigma(\mathrm U_p^n)]_0}[q(b)-1]^2}\]
\end{corollary}
\begin{lemma}\emph{For $k\leq n$ and for $b\sim [\Sigma(\mathrm U_p^n)]_0$, we have that}
\[\EE\limits_{b\sim [\Sigma(\mathrm U_p^n)]_0}[q(b)-1]^2 = \sum\limits_{k=1}^n \frac{\EE[K_k(|s|_0)]^2}{{n\choose k}(p-1)^{n-k}}\]
\end{lemma}
\begin{proof} \Cref{sec:generalizing} says that $q(b)$ has a unique Krawtchouk expansion, where the coefficients on each Krawtchouk basis are given by \Cref{prop: 2}, we observe that:
\begin{equation*}
    \EE\limits_{b\sim [\Sigma(\mathrm U_p^n)]_0}[q(b)-1]^2 = \EE\limits_{b\sim [\Sigma(\mathrm U_p^n)]_0} \bigg[\bigg(\sum\limits_{k=0}^n \hat{q}(k) K_k(b) - 1\bigg)^2\bigg]
\end{equation*}
Then, recall that $\hat{q}(k) = \frac{\EE[K_k(|s|_0)]}{{n\choose k}(p-1)^{n-k}}$. So, \[\hat{q}(0) = \frac{\EE[K_0(|s|_0)]}{(p-1)^n}=\frac{1}{(p-1)^n}.\] Similarly, by the definition of the Krawtchouk function and the reciprocity relation $\frac{K_k(\ell)}{{n\choose k}(p-1)^{n-k}}=\frac{K_s(\ell)}{{n\choose s}(p-1)^{n-s}}$, we have that $K_0(b) = K_n(b)=(p\!-\!1)^n$. Therefore, $\hat{q}(0)K_0(b)=1$. Hence, the above equation simplifies to:
    \[\EE\limits_{b\sim [\Sigma(\mathrm U_p^n)]_0}[q(b)-1]^2 = \EE\limits_{b\sim [\Sigma(\mathrm U_p^n)]_0} \bigg[\bigg(\sum\limits_{k=1}^n \hat{q}(k) K_k(b) \bigg)^2\bigg]\]
Since the generalized Krawtchouk functions are orthogonal (as proven in Lemma 3.2), the product of the non-diagonal entries in the above term all evaluate to $0$. Thus, counting the residuals, we have that the square of the summation is just the summation of the squared terms that it contains. Thus, exploiting the orthogonality of the generalized Krawtchouk functions and the linearity of the expectations, we write that:
    \begin{align*}
    \EE\limits_{b\sim [\Sigma(\mathrm U_p^n)]_0}[q(b)-1]^2 &= \EE\limits_{b\sim [\Sigma(\mathrm U_p^n)]_0} \bigg[\sum\limits_{k=1}^n \hat{q}(k)^2 K_k(b)^2\bigg] \\
    &=  \sum\limits_{k=1}^n \hat{q}(k)^2 \EE\limits_{b\sim [\Sigma(\mathrm U_p^n)]_0} [K_k(b)^2] \quad \text{(linearity of expectations)}\\
    &=  \sum\limits_{k=1}^n \frac{\EE[K_k(|s|_0)]^2}{{n\choose k}^2 (p-1)^{2n-2k}} \cdot \EE\limits_{b\sim [\Sigma(\mathrm U_p^n)]_0} [K_k(b)^2]     \end{align*}
    Finally, we use Lemma 3.2 to write $\EE\limits_{b\sim [\Sigma(\mathrm U_p^n)]_0} [K_k(b)^2]$ as $\langle K_k, K_k \rangle = {n\choose k}$.
    \begin{align*}
    \EE\limits_{b\sim [\Sigma(\mathrm U_p^n)]_0} [q(b)-1]^2 &=  \sum\limits_{k=1}^n \frac{{n\choose k}\EE[K_k(|s|_0)]^2}{{n\choose k}^2 (p-1)^{2n-2k}} =  \sum\limits_{k=1}^n \frac{\EE[K_k(|s|_0)]^2}{{n\choose k} (p-1)^{2n-2k}} \qedhere
    \end{align*}
\end{proof}

\begin{theorem} \label{theorem 5.3} \emph{For $\lambda \leq 0.27$},
$\mathrm{TVD}([\Sigma(S(n,p,\lambda))]_0, [\Sigma(\mathrm U_p^n)]_0) \leq O(\lambda)$.\end{theorem}
\begin{proof}
Substituting the result of Corollary 5.2 into the equation derived in Lemma 4.2.1, and scaling the indexes of the summation, we have that:
\begin{align*}
    \EE\limits_{b\sim [\Sigma(\mathrm U_p^n)]_0} [q(b)-1]^2 &= \sum\limits_{k=1}^{n/p} \frac{1}{{n\choose pk}(p\!-\!1)^{2n-2pk}}\bigg(\sum\limits_{d=pk}^{n-pk} \frac{1}{p^k}{{d-1}\choose {\floor{\frac{kp}{p-1}}}-1}{{n-d}\choose{\floor{\frac{kp}{p-1}}}} \lambda^d\bigg)^2 \\
    &= \frac{1}{p^{2k}}\sum\limits_{k=1}^{n/p} \frac{1}{{n\choose pk}(p\!-\!1)^{2n-2pk}}\bigg( \sum\limits_{d=pk}^{n-pk} {{d-1}\choose {\floor{\frac{k}{1-\frac{1}{p}}}}-1}{{n-d}\choose{\floor{\frac{k}{1-\frac{1}{p}}}}} \lambda^d\bigg)^2 \\
    &\leq \frac{1}{p^{2k}} \sum\limits_{k=1}^{n/p} \frac{{n\choose k}^2}{{n\choose pk}(p\!-\!1)^{2n-2pk}}\bigg( \sum\limits_{d=pk}^{n-pk} {{d-1}\choose {k-1}}\lambda^d\bigg)^2\\
    &\leq \frac{1}{p^{2k}} \sum\limits_{k=1}^{n/p} \frac{{n\choose k}^2}{{n\choose pk}}\bigg( \sum\limits_{d=pk}^{n-pk} {{d-1}\choose {k-1}}\lambda^d\bigg)^2
\end{align*}
Note the following generating function relation that $(\frac{x}{1-x})^k = \sum\limits_{m \geq k} {{m-1}\choose{k-1}}x^m$. Then,
\begin{align*}
     \EE\limits_{b\sim [\Sigma(U_p^n)]_0} [q(b)-1]^2 &\leq \frac{1}{p^{2k}} \sum\limits_{k=1}^{n/p} \frac{{n\choose k}^2}{{n\choose pk}}\bigg(\frac{\lambda}{1-\lambda}\bigg)^{2k}\\
     &\leq \frac{1}{p^{2k}}\sum\limits_{k=1}^{n/p}   \bigg(\frac{pk}{n}\bigg)^{pk}\bigg(\frac{en}{k}\bigg)^{2k} \bigg(\frac{\lambda}{1-\lambda}\bigg)^{2k} \quad \text{(From Appendix \ref{claim: bound})}
\end{align*}
\begin{align*}
&= \sum\limits_{k=1}^{n/p} \bigg(\frac{pk}{n}\bigg)^{pk-2k} \bigg(\frac{e \lambda}{1-\lambda}\bigg)^{2k} \\
&\leq \sum\limits_{k=1}^{n/p} \bigg(\frac{e \lambda}{1-\lambda}\bigg)^{2k} \\
&\leq  O(\lambda^2), \quad \text{(for  $\lambda \leq \frac{1}{1+e}$ by geometric sums)}
\end{align*}

Therefore, for $\lambda\leq\frac{1}{1+e}\approx 0.27$, we have that TVD $\leq \sqrt{\EE\limits_{b\sim [\Sigma(U_p^n)]_0}[p(b)-1]^2} \leq O(\lambda)$. \qedhere 
\end{proof}

\section{Proof Strengths and Limitations}\label{sec:proofeval}

When $\lambda>0.27$, our proof method fails to provide the desired $O(\lambda)$ total variation distance since $\sum_{k=1}^{n/p} (\frac{e\lambda}{1-\lambda})^{2k}$ does not converge and goes to $\infty$. We do, however, reach a higher lower-bound on the radius of convergence ($\lambda \leq 0.27$) than \cite{guruswami_et_al:LIPIcs.ITCS.2021.48}'s $\lambda \leq 0.16$ and \cite{golowich_et_al:LIPIcs.CCC.2022.27}'s more general result but which is only valid for $\lambda < 0.01$ and a fixed $p$ in their interpretation of their $O(\lambda p^{O(p)})$ result, whereas our methodology allows us to remove the dependency of $p$ in the total variation distance (though only for the generalized sticky random walk). We conjecture that $\lambda \leq 0.27$ is not the optimal radius of convergence for the generalized sticky random walk and leave this as an open problem for future research directions in this topic to resolve.

\section{Reductions}\label{sec:reduction}

\begin{theorem} \label{theorem 7.1}
\emph{Let $p$ be the number of vertices in the generalized sticky random walk. Then, for all $p \mod k \equiv 0$, $S(n,p,\lambda)$ reduces from $S(n, k, p\lambda(1-\frac{1}{k}))$.}
\end{theorem}
\begin{proof}
    Consider a generalized sticky random walk which details an irreducible homogeneous Markov chain on $p$ states where the probability of staying at the same state is $\frac{1}{p} + (p-1)\lambda$ and the probability of switching states is $\frac{1}{p} -\lambda$. Then, consider a 'grouped' random-walk, for a grouping of states $V = [p] = V_0 \sqcap \dots \sqcap V_{k-1}$, where $V_i$ contains an arbitrary selection of $p/k$ vertices and where $p \mod k\equiv 0$. Then, $S(n,p,\lambda)$ details an $n$-step long random walk on the generalized sticky random walk. Then, we note that the probability that the current state of the grouped random walk stays at itself must be $(\frac{1}{k} + (p-1)\lambda) + (\frac{1}{p}-\lambda)(\frac{p}{k}-1) = \frac{1}{k} + p\lambda(1-\frac{1}{k})$. Here, the first term comes from the probability of any state $E$ in the sticky random walk staying at itself, and the second term comes from the probability that any other vertex in the same group as $E$ transitions to $E$. Since this is true for any of the grouped vertex, we conclude that our grouping of the random walk yields a sticky random walk on $k$ vertices and bias $p\lambda(1-\frac{1}{k})$, which completes the reduction. \qedhere
\end{proof}

\begin{theorem} \label{theorem 7.2} \emph{Every generalized sticky random walk $S(n,p,\lambda)$ corresponds to a $n$-step long random walk on a $p\lambda$-spectral expander with $p$ vertices.}
\end{theorem}
\begin{proof}
    We first extend Definition 45 of the sticky random walk matrix in  \cite{golowich_et_al:LIPIcs.CCC.2022.27}. For subsets $A, B \in [p]$, let $J_{A,B}\in \RR^{A\times B}$ denote the matrix with all entries equal to $\frac{1}{|A|}$. Then, for $\lambda \in [0,1]$, let $G_{\lambda,p} \in \RR^{p\times p}$ denote the generalized sticky random walk matrix. Then, by definition, we write that:
    \[G_{\lambda,p} = (1-\lambda)\mathrm J_{V,V} + p\lambda \mathrm I_{\{n\times n\}}\]
    Then, note that $\|I_{n\times n}\|_2$ is $1$, as it acts as $J$ on the orthogonal subspaces $\RR$ of $\RR^p$. Therefore, $\lambda(G_{\lambda,p}) \leq p \lambda$ since the eigenvector of $G_{\lambda,p}$ is orthogonal to $J_{V,V}$ and so $G_{\lambda,p}\nu = ((1-\lambda)\mathrm J_{v,v} + p\lambda I_{\{n\times n\}})\nu = (1-\lambda)\mathrm J_{v,v}\nu + p\lambda I_{\{n\times n\}}\nu = p\lambda \nu$. The opposite inequality comes from the fact that $\frac{n-1}{n}\vec{\mathbbm{1}}_{\{c_0\}} - \frac{1}{n}\sum_{i=1}^{p-1} \vec{\mathbbm{1}}_{\{c_i\}} \in \vec{\mathbbm{1}}^\perp$ is an eigenvector of $G_{\lambda,p}$ with eigenvalue $p \lambda$, which proves the theorem.\qedhere
\end{proof}

\section{Acknowledgements}
This work was done as part of a Caltech undergraduate thesis, and this preprint submission has been reformatted from the thesis submission \cite{Anand2023}. EA thanks Professor Chris Umans for his mentorship during this project, and Elia Gorokhovsky and David Hou for helpful suggestions over the course of this work.

\newpage

\bibliographystyle{alpha}
\bibliography{main}

\newpage


\begin{appendix}
\section{Appendix: Technical Lemmas}\label{technical_lemmas}
The following technical lemmas are used at various points in the paper.\\

\begin{claim}\label{claim: bound}
\emph{For $1\leq k\leq \frac{n}{p}$, we can bound $\frac{{n\choose k}^2}{{n\choose {pk}}}\leq (\frac{ne}{k})^{2k} \cdot (\frac{pk}{n})^{pk}$.}
\end{claim}
\begin{proof}
We use the bound that $(\frac{n}{k})^k \leq {n\choose k} < (\frac{ne}{k})^k$. This gives us that $\frac{{n\choose k}^2}{{n\choose {pk}}} \leq \frac{(\frac{ne}{k})^{2k}}{(\frac{n}{pk})^{pk}}$. Simplifying it yields that $(\frac{ne}{k})^{2k} \cdot (\frac{pk}{n})^{pk}$, which proves the claim. \qedhere
\end{proof}

\vspace{0.2cm} 

\begin{definition} \emph{For $p\geq 2$, let $\omega_p$ denote the $p^{th}$ primitive root of unity if it satisfies $(\omega_p)^p=1$, and if there does not exist $q\in\NN$ where $q<p$ such that $(\omega_p)^q=1$. Specifically, the multiplicative order of the $p^{th}$ primitive root of unity must be $p$. Then, for $1\leq k<p$, we must have that}
\[\sum\limits_{j=0}^{p-1} (\omega_p)^{k j} = \omega_p^{k\cdot 0} + \omega_p^{k\cdot 1} + ... + \omega_p^{k\cdot(p-1)} =0\]
\end{definition}
\begin{proof}
Given $\omega_p^p=1$, it is clear that for $k<p$, $(\omega_p^k)^p=1$. Thus, \[(\omega_p^k)p-1=0=(\omega_p^k-1)(\omega_p^{k\cdot 0} + \omega_p^{k\cdot 1} + ... + \omega_p^{k\cdot(p-1)}).\]
Since $\omega_p^k\neq 1$ as $\omega_p$ is a primitive root of unity, we must have that $\omega_p^{k\cdot 0} + \omega_p^{k\cdot 1} + ... + \omega_p^{k\cdot(p-1)}=0$, which proves the claim.
\end{proof}

\vspace{0.2cm}

\begin{lemma} \emph{Orthogonality of the Krawtchouk functions:}
\[\langle K_r, K_s \rangle = \EE_{b\sim\text{Bin}(n,\frac{1}{2})}[K_r(b)K_s(b)] = \begin{cases}
0 \quad \text{if } r\neq s \\
\genfrac(){0pt}{1}{n}{s} \quad \text{if } r=s\end{cases}\]
\end{lemma}
\begin{proof}
We provide a probabilistic interpretation of the Krawtchouk function as demonstrated in \cite{Samorodnitsky98onthe}. Fix $A\in \genfrac(){0pt}{1}{[n]}{\ell}$, $ B\in_U \genfrac(){0pt}{1}{[n]}{s}$, and choose $C\in_U \genfrac(){0pt}{1}{[n]}{r}$. Then,
\[K_s(\ell) = \genfrac(){0pt}{0}{n}{s} \EE[(-1)^{|A \cap B|}], \quad K_r(\ell) = \genfrac(){0pt}{0}{n}{r} \EE[(-1)^{|A\cap C|}]\]
The inner product of $K_r$ and $K_s$ is then:
\begin{equation*}\begin{split}
\langle K_r, K_s\rangle &= \genfrac(){0pt}{0}{n}{r} \genfrac(){0pt}{0}{n}{s} \EE[(-1)^{|A\cap B|}]\EE[(-1)^{|A\cap C|}]\\
&= \genfrac(){0pt}{0}{n}{r} \genfrac(){0pt}{0}{n}{s} \EE[(-1)^{|A\cap B|}(-1)^{|A\cap C|}] \quad \text{(A, B, and C are independent)}
\end{split}\end{equation*}
If $B \neq C$, then $\EE[(-1)^{|A\cap B|}(-1)^{|A\cap C|} | B \neq C]=0$, since for each $x\in B\Delta C$, we could either have $x\in A$ or $x\notin A$, which contributes a +1 or -1 (or vice versa) (respectively) to the expectation. So, the expected value must be 0. Conversely, if $B=C$, then it must be the case that $r=s$, which occurs with probability $\frac{1}{\genfrac(){0pt}{1}{n}{s}}$. So, $\EE[(-1)^{|A\cap B|}(-1)^{|A\cap C|}]=\EE[(-1)^{2|A\cap B|}]=\frac{1}{{n\choose s}}$. This yields that
\[\langle K_r, K_s \rangle = \begin{cases}
0 \quad \text{if } r\neq s \\
\genfrac(){0pt}{1}{n}{s} \quad \text{if } r=s \end{cases}\qedhere\]
\end{proof}

\newpage
\end{appendix}

\section{Appendix: Generalizing the Krawtchouk Functions into the Complex Domain}
\label{appendix 2}
In this section of the Appendix, we also present a generalization of the Krawtchouk polynomials into the complex domain that yields an elegant method of analysis for bounding the total variation distance of $[\Sigma(S(n,p,\lambda))]_0$ and $[\Sigma(\mathrm U_n^p)]_0$, but only for when $p$ is a prime number. For the generalized sticky random walk, this method yields an upper-bound on TVD$([\Sigma(S(n,p,\lambda))]_0, [\Sigma(\mathrm U_n^p)]_0)$ of $O(\lambda p^{O(p)})$, which matches the upper-bound derived in Corollary 4 of \cite{golowich_et_al:LIPIcs.CCC.2022.27} through Fourier-analytic means, which the body of our paper shows to be suboptimal when the size of the alphabet used is allow to increase asymptotically. Nonetheless, we include a proof of this generalization to present our novel treatment of the Krawtchouk function.\\

\begin{definition}
\emph{Given any set $S_k \subseteq [n]$ with cardinality $|S_k|=k$, let $v_{S_k}$ denote a bit-string in $\{0, 1\}^n$ such that for each $i \in S_k$, $(v_{S_k})_i=0$ and for each $i \in [n] \setminus S_k$, $(v_{S_k})_i = 1$. Similarly, given the context of a prior prime number $p$, let $w_{S_k}$ denote a string in $\ZZ_p^n$ where for each $i \in {S_k}$, $(w_{S_k})_i = 0$, and for each $i \in [n]\setminus S_k$, $(w_{S_k})_i \in [p-1]$}.\\ \end{definition}

\begin{definition}\label{definition 11} \emph{Let $D$ denote the `zero distribution' where for any $\ell \in [n]$, $\Pr[\ell]$ denotes the probability that a string $s \in \ZZ_p^n$ has $\ell$ zeros. Specifically, $\Pr[\ell] = \frac{{n \choose \ell} (p-1)^{n-\ell}}{p^n}$, since there are ${n\choose\ell}$ ways to select the locations for the $\ell$ 0s in a string of length $n$, and $(p-1)^{n-\ell}$ to populate the other $n-k$ locations with characters from $[p-1]$.}\\ \end{definition}

\begin{definition}\emph{For all $p\in\NN$, consider $\omega_p$, the p-th principal root of unity. Then, the generalized Krawtchouk function $K_k(\ell)$, for any $\alpha$ where $\alpha$ has $\ell$ 0s and $n-\ell$ 1s, is defined as:}
\begin{equation*}
    K_k(\ell) = \sum_{\substack{y\in \ZZ_{p}^n \\ \ |y|_0 = k}} (\omega_p)^{\alpha \cdot y}
\end{equation*}
\end{definition}

\begin{lemma}
\emph{Orthogonality of the generalized Krawtchouk function: The generalized Krawtchouk functions form an orthogonal basis of the functions mapping $\ZZ_{n+1} \rightarrow\RR$ (for the distribution $D$ as described in \Cref{definition 11}) with respect to the inner product $\langle f,g \rangle = \EE\limits_{b \sim D}[f(b) g(b)]$.}
\[\langle K_r, K_s\rangle = \begin{cases}
        0, & \text{if } r\neq s \\
        {n \choose r} (p-1)^{n-r}, & \text{if } r=s
        \end{cases}\]
\end{lemma}
\begin{proof}
In a similar vein to the proof of Lemma 9.1 in \Cref{technical_lemmas}, we start by providing a probabilistic interpretation of the generalized Krawtchouk function. Fix $A\in \genfrac(){0pt}{1}{[n]}{\ell}$, and choose $ B\in_U \genfrac(){0pt}{1}{[n]}{r}$ and $C\in_U \genfrac(){0pt}{1}{[n]}{s}$. This is equivalent to fixing a string $v_{A_\ell} \in \ZZ_p^n$ where $|v_{A_\ell}|_0=\ell$ and $(v_{A_\ell})_t = 0$ for all $t\in A$, and randomly choosing $w_{B_r}, w_{C_s} \in \ZZ_p^n$ where $|w_{B_r}|_0=r$ and $|w_{C_s}|_0=s$. Then,
\[K_s(\ell) = \genfrac(){0pt}{0}{n}{s} (p-1)^{n-s} \mathbb{E}_{B_r}[\omega_p^{\langle v_{A_\ell}, w_{B_r}\rangle}], \quad K_r(\ell) = \genfrac(){0pt}{0}{n}{r} (p-1)^{n-r} \EE_{C_s}[\omega_p^{\langle v_{A_\ell}, w_{C_s}\rangle}]\]
The inner product of $K_r$ and $K_s$ (with respect to the probability distribution $D$) is then:
\begin{align*}
\langle K_r, K_s\rangle &= \sum\limits_{\ell=0}^n \Pr[\ell] K_r(\ell) \overline{K_s(\ell)} \\
&= \sum\limits_{\ell=0}^n \genfrac(){0pt}{0}{n}{\ell}\genfrac(){0pt}{0}{n}{r} \genfrac(){0pt}{0}{n}{s} \frac{(p-1)^{n-\ell}}{p^n} (p-1)^{2n-s-r}\EE\limits_{B_r}[\omega_p^{\langle v_{A_\ell}, w_{B_r}\rangle}] \EE\limits_{C_s}[\overline{\omega}_p^{\langle v_{A_\ell}, w_{C_s}\rangle}]\\
&= \genfrac(){0pt}{0}{n}{r} \genfrac(){0pt}{0}{n}{s} \frac{(p-1)^{2n-s-r}}{p^n} \sum\limits_{\ell=0}^n \genfrac(){0pt}{0}{n}{\ell} (p-1)^{n-\ell}\EE\limits_{B_r, C_s}[\omega_p^{\langle v_{A_\ell}, w_{B_r} \rangle-\langle v_{A_\ell}, w_{C_s}\rangle}] \quad \text{(A, B, and C are independent)}\\
&= \genfrac(){0pt}{0}{n}{r} \genfrac(){0pt}{0}{n}{s} \frac{(p-1)^{2n-s-r}}{p^n} \sum\limits_{\ell=0}^n \genfrac(){0pt}{0}{n}{\ell} (p-1)^{n-\ell}\EE\limits_{B_r, C_s}[\omega_p^{\langle v_{A_\ell}, w_{B_r} - w_{C_s}\rangle \mod p}]
\end{align*}
The last line in the derivation follows by the bilinearity of the inner product. If $r\neq s$, then $B_r\neq C_s$ and $\EE[\omega_p^{\langle v_{A_\ell},  w_{B_r}-w_{C_s}\rangle \mod p}]=\EE[\omega_p^{\langle v_{A_\ell},  w_{B_r}\rangle \mod p}]$ since the distribution of $B-C$ is uniformly random (as are the distributions of $B$ and $C$). Therefore, the distribution of $B-C$ must be indistinguishable from the distribution of $B$ under the expectation operator. Additionally, $\EE[\omega_p^{\langle A,B\rangle\mod p}]=\sum_{j=0}^p\omega_p^{kj}=0$. If $B=C$, (which is only when $r=s$), the associated probability of it occurring must be  $\frac{1}{{n\choose r} (p-1)^{n-r}}$. \\

In this case, $\langle K_r, K_s\rangle = {n \choose s}{n\choose r} (p-1)^{2n-r-s}\EE[\omega^{\langle A,0 \rangle}] = {n \choose s} (p-1)^{n-r}$.
    \[
        \EE\limits_{b\sim D}[K_r(b) K_s(b)] = \begin{cases}
        0, & \text{if } r\neq s \\
        {n \choose r} (p-1)^{n-r}, & \text{if } r=s
        \end{cases}
    \qedhere\]
\end{proof}    

\vspace{0.2cm}

\begin{lemma}\emph{The generalized Krawtchouk function $K_k(\ell)$ is invariant against choices of $\alpha\in\{0,1\}^n$, where $|\alpha|_0=\ell$. Specifically, we have that for a fixed $B\in \ZZ_p^n$ where $|B|_0=k$, that}
\end{lemma}
\[\EE_{\substack{\text{Fixed }A \in \ZZ_2^n \\|A|_0=\ell} }[\omega_p^{\langle A,B\rangle}] = \EE_{\substack{\text{Random }A' \in \ZZ_2^n \\|A'|_0=\ell} }[\omega_p^{\langle A',B\rangle}]\]
\begin{proof}
The problem is equivalent to showing that 
\[\omega_p^{\langle A,B\rangle} \cdot \frac{1}{\EE\limits_{A'}[\omega_p^{\langle A',B\rangle}]} = \omega_p^{\langle A,B\rangle} \cdot \EE\limits_{A'}\bigg[\frac{1}{\omega_p^{\langle A',B\rangle}}\bigg] \overset{?}{=} 1\]
By the independence of $A, A'$, and $B$, the problem reduces to showing that
\[\EE\limits_{A'}\bigg[\frac{\omega_p^{\langle A,B\rangle}}{\omega_p^{\langle A',B\rangle}}\bigg] = \EE\limits_{ A'}[\omega_p^{\langle A-A',B\rangle}] \overset{?}{=} 1\]
Since $A'$ is uniformly random in $\ZZ_2^n$, $A-A'$ must also be uniformly random for a fixed $A$. So, the problem reduces further to showing (under the same conditions of $B$ and $A'$) that $\EE[\omega_p^{\langle A', B\rangle} ]\overset{?}{=} 1$. Since $\sum_{j=0}^p \omega_p^{k j}=0$ from definition 2.5, we must have that $\sum_{j=0}^{p-1}\omega_p^{kj}=1$. Thus, $\EE[\omega_p^{\langle A', B\rangle} ]= \sum_{j=0}^{p-1}\omega_p^{kj}=1$, which proves the claim.
\end{proof}
\vspace{0.2cm}

\begin{corollary}
    \emph{The orthogonality of the generalized Krawtchouk function implies a reciprocity relation:}
\[\frac{K_k(\ell)}{{n\choose k}(p-1)^{n-k}} = \frac{K_s(\ell)}{{n\choose s}(p-1)^{n-s}}\]
\end{corollary}

\vspace{0.2cm}

We now provide a brief calculation for the expectation of the generalized Krawtchouk function, since it is necessary for bounding the total variational distance between $[\Sigma(S(n,p,\lambda))]_0$ and $[\Sigma(U_p^n)]_0$. \\

\begin{definition}
\emph{Given any set $T \subseteq[n]$ such that $|T|=k$, let $a_1 < ... < a_k$ be the elements of $T$ in increasing order. Then, for any $c\in \ZZ_p$, let
\[\mathrm{shift}_c(T) = \sum\limits_{i=0}^{\floor{|k-c|/p}}(a_{c+i p}-a_{c+i p-1})\]
Then, for any $c$ such that $k \mod p = -c$, and for any $d\in \ZZ_{n+1}$, let $\phi_c(d)$ denote the number of subsets of $[n]$ of size $k$ such that $\mathrm{shift}_c(T)=d$. Note that for any $t\leq 0, a_t=0$.}\\
\end{definition}

\begin{lemma} \emph{The expectation of the Krawtchouk function is}: 
\[\EE[K_k(|s|_0)]=\begin{cases}(p-1)^{n-k} \sum\limits_{d=k}^{n-k}\phi_0(d)\lambda^d, & c= k\!\!\!\!\mod p \equiv 0 \\ 0, & c= k\!\!\!\!\mod p \not\equiv 0\end{cases}\]
\end{lemma}
\begin{proof}
\begin{align*}
\EE[K_k(|s|_0)] = \sum\limits_{s\sim S(n,p,\lambda)} \Pr[s] \sum_{\substack{y\in \ZZ_p^n \\ |y|_0 = k}} \omega_p^{y\cdot s} = \sum\limits_{s\sim S(n,p,\lambda)} \Pr[s] \sum_{\substack{y\in \ZZ_p^n \\ |y|_0 = k}} \omega_p^{\sum\limits_{i=1}^n y_i\cdot s_i} = \sum\limits_{s\sim S(n,p,\lambda)} \Pr[s] \sum_{\substack{y\in \ZZ_p^n \\ |y|_0 = k}} \prod\limits_{i=1}^n\omega_p^{y_i\cdot s_i}
\end{align*}
The dot-product on the exponent of $\omega$ only takes the summation of the element-wise product of $\alpha$ and $y$ for positions on $y$ that are non-zero. We can rewrite this summation by considering the location of non-zero terms in $y$. So, for $T=\{a_1 < ... < a_{n-k}\}$:
\[\EE[K_k(|s|_0)]=\sum\limits_{s\sim S(n,p,\lambda)} \Pr[s] \sum\limits_{T\in{{[n]}\choose {n-k}}} \sum\limits_{\beta \in [p-1]^{n-k}} \prod\limits_{i\in T}\omega_p^{\beta_i s_i}\]
Further, choosing a $T\in {{[n]}\choose {n-k}}$ implies a choice of $\overline{T} = {{[n]}\choose k}=[n]\setminus T$. Hence, the summation reduces to:\\
\begin{equation*}\EE[K_k(|s|_0)]=\sum\limits_{s\sim S(n,p,\lambda)} \Pr[s] \sum\limits_{T\in{{[n]}\choose k}} \sum\limits_{\beta \in [p-1]^{n-k}} \prod\limits_{i\in \overline{T}}\omega_p^{\beta_i s_i}= \sum\limits_{T\in{{[n]}\choose k}} \sum\limits_{\beta \in [p-1]^{n-k}} \EE\limits_{s\sim S(n,p,\lambda)} \bigg[ \prod \limits_{i\in \overline{T}}\omega_p^{\beta_i s_i}\bigg]\end{equation*}

Next, observe that the sticky random walk is a Markov chain where $\omega_p^{s_i}=\omega_p^{s_{i-1}}$ with probability $\frac{1}{p} + (p-1)\lambda$. We can instead model the transitions of strings from the sticky random walk as random variables $u$, where $u_1$ is uniformly distributed in $\ZZ_p$ and for $i\geq 2$, $u_i$ is uniformly distributed on $(1-\lambda)U[\ZZ_p]+\lambda \cdot \mathbbm{1}_0$. Intuitively, $\lambda$ is the additional probability of staying on the same vertex. So, for each $s\!\sim\!S$, we refactor $s$ to $\Tilde{s}=\{\Tilde{s}_1, ..., \Tilde{s}_n\}$, where $\Tilde{s}_i = \mathbbm{1}\{s_i \neq 0\}$. Note then that since $\beta_i$ is uniformly random in $[p-1]$, that $\beta_i s_i \mod p$, and therefore $\beta_i \tilde{s}_i \mod p$ must also be uniformly random in $[p-1]$. Therefore, $\beta_i \tilde{s}_i$ and $\beta_i s_i$ are both uniformly random in $[p-1]$ and have the same distributions. Hence, we write that:
\begin{equation*}
    \begin{split}
        \EE\limits_{\beta, s}\bigg[\prod\limits_{i\in \overline{T}} \omega_p^{\beta_i s_i}\bigg] = \EE\limits_{\beta, s}\bigg[\prod\limits_{i\in \overline{T}} \omega_p^{\beta_i \Tilde{s}_i}\bigg] &= \EE\limits_{\beta, s}\bigg[\omega_p^{\sum\limits_{i\in \overline{T}}\beta_i \sum\limits_{j=1}^i u_j}\bigg] = \prod\limits_{j=1}^{a_{n-k}} \EE\limits_{\beta, s}\bigg[\omega_p^{\sum\limits_{i\in \overline{T}; i\geq j}\beta_i u_j}\bigg]
    \end{split}
\end{equation*}

Since $u_j$ is random, the distribution of $\beta_i u_j \mod p$ must also be random, and therefore indistinguishable from the distribution of $u_j$. Therefore, $\beta_i u_j \!\!\mod p$ and $u_j \!\!\mod p$ are invariant under the expectation operation of its exponentiation under $\omega_p$. So,
\begin{equation*}
    \begin{split}
        \EE\limits_{\beta, s}\bigg[\prod\limits_{i\in \overline{T}} \omega_p^{\beta_i s_i}\bigg] &= \prod\limits_{j=1}^{a_{n-k}} \EE\bigg[\omega_p^{\sum\limits_{i\in \overline{T}; i\geq j} u_j}\bigg]
\end{split}
\end{equation*}
When $j=1$, the above definition directly implies that:
\[\EE\bigg[\omega_p^{\sum\limits_{i\in \overline{T}; i\geq 1} u_1}\bigg] = \EE[\omega_p^{|\overline{T}|u_1}] = \begin{cases}
1, \quad\text{if }|\overline{T}| \mod p \equiv 0 \\
0, \quad\text{otherwise}\end{cases}\]
Conversely, when $j\geq 2$, let $T_j = \{i\in \overline{T}; i\geq j\}$. Then, 
\[\EE\bigg[\omega_p^{\sum\limits_{i\in \overline{T}; i \geq j}u_j}\bigg] = \EE[\omega_p^{|T_j|u_j}] = \begin{cases}
1, \quad\quad\quad \text{if } |T_j| \mod p \equiv 0 \\
\EE[\omega_p^{u_j}], \quad \text{otherwise}\\
\end{cases}\]

Next, observe that for $j\geq 2$, $\EE[\omega_p^{u_j}]=\lambda$.
\begin{proof} We write $u_j$ in terms of our refactoring (similarly to \Cref{sec:expectations}) and take the expectation:
\begin{align*}
\EE[\omega_p^{u_j}] &= \EE[\omega_p^{(1-\lambda)U[\ZZ_p]+\lambda\cdot\mathbbm{1}_0}] 
= \sum\limits_{k=0}^{p-1} \omega_p^{k} \Pr[u_j=k]\\
&= \omega_p^0 \Pr[u_j=0] + \omega_p^1 \Pr[u_j=1] + ... + \omega_p^{p-1} \Pr[u_j=p-1] \\
&= \omega_p^0 \bigg(\frac{1}{p} + \lambda\bigg(\frac{p-1}{p}\bigg)\bigg) + \omega_p^1 \bigg(\frac{1}{p} - \frac{\lambda}{p}\bigg) + ... + \omega_p^{p-1} \bigg(\frac{1}{p} - \frac{\lambda}{p}\bigg)\\
&= \frac{1}{p}\sum\limits_{k=0}^{p-1}\omega_p^k - \frac{\lambda}{p}\sum\limits_{k=0}^{p-1}\omega_p^k + \lambda\omega_p^0\\
&= \lambda \qedhere
\end{align*}\end{proof}
Then, for $k \mod p \equiv 0$, we have that
\begin{equation*}\begin{split}
    \EE[K_k(|s|_0)] &= \sum\limits_{T\in{[n]\choose k}}\sum\limits_{\beta\in[p-1]^{n-k}}\EE\limits_{\substack{s\sim S}}\bigg[\prod\limits_{i\in \overline{T}}\omega_p^{\beta_i s_i}\bigg] \\
    &= (p\!-\!1)^{n-k} \!\!\!\sum\limits_{T\in{[n]\choose k}}\EE\limits_{\substack{ s\sim S \\ \beta\in[p-1]^{n-k}}}\bigg[\prod\limits_{i\in \overline{T}}\omega_p^{\beta_i s_i}\bigg] \\
    &= \!(p\!-\!1)^{n-k}\!\!\! \sum\limits_{T\in{[n]\choose k}}\prod\limits_{j=1}^{a_{n-k}}\lambda \\
    &= (p\!-\!1)^{n-k}\!\!\! \sum\limits_{T\in{[n]\choose k}}\lambda^{a_{n-k}}
\end{split}
\end{equation*}
We then parameterize the summation over every possible value of the shift of T (for $k\!\!\mod p \equiv 0$):
\[\EE[K_k(|s|_0)] = (p\!-\!1)^{n-k} \sum\limits_{d=k}^{n-k} \bigg(\sum_{\substack{T\in{[n]\choose k}\\ \text{shift}_0\text{(T)=d}}} 1\bigg) \lambda^d = (p\!-\!1)^{n-k} \sum\limits_{d=k}^{n-k} \phi_0(d) \lambda^d\]
This yields the claim.\\
\end{proof}

\begin{lemma}\emph{For $c\in\NN$ where $0\leq c\leq p$, and for $d\in\NN$ where $k\leq d\leq n-k$, the number of $k$-sized subsets of $[n]$ that satisfy $\mathrm{shift}_c(T)=d$ is:}
\[\phi_c(d) = \sum_{\substack{T\in{{[n]}\choose k}\\ \mathrm{shift}_c(T)=d}}1 = {{d-1}\choose{\floor{\frac{|k-c|}{p-1}}-1}}{{n-d}\choose {\floor{\frac{|k-c|}{p-1}}}}\]
\end{lemma}
\begin{proof}
To determine $\phi_c(d)$, we count the total number of ways to choose $a_1 < a_2 < ... < a_k$ such that $(a_c-a_{c-1})+(a_{c+p}-a_{c+p-1})+(a_{c+2p}-a_{c+2p-1})+...=d$, where for any $j\leq 0$, $a_j=0$. To do this, we combine each element-wise interval $(a_{c+ip}, a_{c+ip-1})$ to form a contiguous interval of length $d$ (starting from $a_{c-1}=0$). The remaining contiguous region that excludes these intervals must then have a length of $n-d$. We then abstract the number of ways to count $a_1<...<a_k$ by counting the number of intervals that have a length of $d$ when combined, such that the remaining intervals have a length $n-d$.\\

From a length of $d-1$ (accounting for $a_0=0$), we need to select intervals that form a length of $\floor{|k-c|/(p-1)}-1$ since they represent the number of choices of elements of $T$ that are index-separated by $p$. Similarly, from a length of $n-d$, we need to select intervals that form a length of $\floor{|k-c|/(p-1)}$ possible intervals, since they represent every other element of $T$. This second constraint is to ensure that the total length of the intervals chosen is exactly $n$. Hence, we write that:
\[\phi_c(d) = \sum_{\substack{ T\in{{[n]}\choose k}\\\text{shift}_c(T)=d}}1 = {{d-1}\choose{\floor{\frac{|k-c|}{p-1}}-1}}{{n-d}\choose {\floor{\frac{|k-c|}{p-1}}}} \qedhere\]
\end{proof}

\begin{corollary}\emph{By combining the results from Lemmas 5.1 and 5.2, we have that the expectation of the Krawtchouk function is}: \[\EE[K_k(|s|_0)]=\begin{cases}(p-1)^{n-k} \sum\limits_{d=k}^{n-k} {{d-1}\choose {\floor{\frac{k}{p-1}}}-1}{{n-d}\choose{\floor{\frac{k}{p-1}}}} \lambda^d, & k\!\!\!\!\mod p \equiv 0 \\ 0, &  k\!\!\!\!\mod p \not\equiv 0\end{cases}\]
\end{corollary}

\begin{proposition}\emph{The orthogonality of the generalized Krawtchouk function $K_k(\ell)$ implies that for any function $f:\ZZ_{n+1}\rightarrow\RR$, there exists a unique expansion $f(\ell) = \sum\limits_{k=0}^n \hat{f}(k) K_k(\ell)$, where for $0\leq k\leq n$}, 
\[\hat{f}(k) = \frac{\EE\limits_{b\sim [\Sigma(U_p^n)]_0}[f(b)K_k(b)]}{{n\choose k}(p-1)^{n-k}}\]\end{proposition}

\begin{definition}
    \emph{Let $q:\ZZ_{n+1}\rightarrow\RR$, where $q(\ell)=\frac{\Pr\limits_{s\sim S(n,p,\lambda)}[|s|_0=\ell]}{{n\choose \ell}(p-1)^{n-\ell}}p^n$}. Intuitively, $q(\ell)$ is the ratio of the probability of getting a string with $\ell$ 0s from the sticky random walk $S(n,\lambda,p)$ to the probability of getting a string with $\ell$ 0s from the uniformly random distribution.
\end{definition}

\textbf{Lemma 3.4: }\emph{Expanding $q(\ell)$ through the generalized Krawtchouk function yields that:}
\[\hat{q}(k) = \frac{1}{{n\choose k}(p-1)^{n-k}}\EE\limits_{s\sim S(n,p,\lambda)}[K_k(|s|_0)]\]
\begin{proof}\begin{align*}
        \hat{q}(k) &= \frac{1}{{n\choose k}(p-1)^{n-k}}\sum\limits_{b=0}^n {n\choose b} \frac{(p-1)^{n-b}}{p^n} q(b) K_k(b) \\
        &= \frac{1}{{n\choose k}(p-1)^{n-k}}\sum\limits_{b=0}^n \Pr\limits_{s\sim S(n,p,\lambda)}[|s|_0=b]K_k(b) \\
        &= \frac{1}{{n\choose k}(p-1)^{n-k}} \EE\limits_{s\sim S(n,p,\lambda)}[K_k(|s|_0)] \qedhere
\end{align*}\end{proof}

\textbf{Lemma 3.5: }\emph{For $s\in S(n,p,\lambda)$, we have that $\Pr[|s|_0=\ell]=\frac{1}{p^n} \sum\limits_{k=0}^n K_\ell(k) \EE\limits_{s\sim S(n,p,\lambda)}[K_k(|s|_0)].$}
\begin{proof}
\begin{align*}
        \Pr[|s|_0=\ell] &= \frac{{n\choose \ell}(p-1)^{n-\ell}}{p^n}  q(\ell) \\
        &= \frac{{n\choose \ell}(p-1)^{n-\ell}}{p^n} \sum\limits_{k=0}^n \hat{q}(k) K_k(\ell) \\
        &= \frac{{n\choose \ell}(p-1)^{n-\ell}}{p^n} \sum\limits_{k=0}^n \frac{K_k(\ell)}{{n\choose k}(p-1)^{n-k}} \EE[K_k(|s|_0)] \\
        &= \frac{1}{p^n} \sum\limits_{k=0}^n \frac{{n\choose \ell}}{{n\choose k}} \frac{(p-1)^{n-\ell}}{(p-1)^{n-k}} \EE[K_k(|s|_0)] K_k(\ell)\\
        &= \frac{1}{p^n} \sum\limits_{k=0}^n K_\ell(k)  \EE\limits_{s\sim S(n,p,\lambda)}[K_k(|s|_0)] \quad \text{(By the reciprocity relation in Corollary 9.3.1)} \qedhere
    \end{align*}
\end{proof}

\begin{lemma}\emph{The total variational distance between the $n$-step generalized sticky random walk on $p$ vertices and the $n$-round uniform distribution on $p$ states is given by} \[\mathrm{TVD}([\Sigma(S(n,p,\lambda))]_0, [\Sigma(U_p^n)]_0)= \frac{1}{2} \EE\limits_{b\sim [\Sigma(U_p^n)]_0}[|p(b)-1|]\]
\end{lemma}
\begin{proof}
\begin{equation*}
    \begin{split}
        \mathrm{TVD}([\Sigma(S(n,p,\lambda))]_0, [\Sigma(U_p^n)]_0) &= \frac{1}{2} \sum\limits_{\ell=0}^n \bigg|\Pr[|s|_0=\ell] - \frac{{n\choose\ell}(p-1)^{n-\ell}}{p^n}\bigg| \\ 
        &= \frac{1}{2} \sum\limits_{\ell=0}^n \bigg|{n\choose\ell}q(\ell)\frac{(p-1)^{n-\ell}}{p^n} - \frac{{n\choose\ell}(p-1)^{n-\ell}}{p^n}\bigg| \\
        &= \frac{1}{2} \sum\limits_{\ell=0}^n \bigg|\frac{{n\choose\ell}}{p^n} (p-1)^{n-\ell}(q(\ell)-1)\bigg| \\
        &= \frac{1}{2} \EE\limits_{b\sim [\Sigma(U_p^n)]_0} [|q(b)-1|] \qedhere
    \end{split}
\end{equation*}
\end{proof}

\begin{claim}\emph{The total variational distance between the generalized sticky random walk and the multinomial distribution has the following upper bound as a result of convexity:}\[\mathrm{TVD}([\Sigma(S(n,p,\lambda))]_0, [\Sigma(\mathrm U_p^n)]_0) \leq \frac{1}{2}\sqrt{\EE\limits_{b\sim [\Sigma(\mathrm U_p^n)]_0}[q(b)-1]^2}\]
\end{claim}

\begin{lemma}\emph{For $k\leq n$ and for $b\sim [\Sigma(U_p^n)]_0$, we have that}
\[\EE\limits_{b\sim [\Sigma(U_p^n)]_0}[q(b)-1]^2 = \sum\limits_{k=1}^n \frac{\EE[K_k(|s|_0)]^2}{{n\choose k}(p-1)^{n-k}}\]
\end{lemma}
\begin{proof}
\begin{equation*}
    \EE\limits_{b\sim [\Sigma(U_p^n)]_0}[q(b)-1]^2 = \EE\limits_{b\sim [\Sigma(U_p^n)]_0} \bigg[\bigg(\sum\limits_{k=0}^n \hat{q}(k) K_k(b) - 1\bigg)^2\bigg]
\end{equation*}
Recall that $\hat{q}(k) = \frac{\EE[K_k(|s|_0)]}{{n\choose k}(p-1)^{n-k}}$. So, $\hat{q}(0) = \frac{\EE[K_0(|s|_0)]}{(p-1)^n}=\frac{1}{(p-1)^n}$. Similarly, by the definition of the Krawtchouk function and the reciprocity relation $\frac{K_k(\ell)}{{n\choose k}(p-1)^{n-k}}=\frac{K_s(\ell)}{{n\choose s}(p-1)^{n-s}}$, we have that $K_0(b) = K_n(b)=(p\!-\!1)^n$. Therefore, $\hat{q}(0)K_0(b)=1$. Thus, the above equation simplifies to:
    \[\EE\limits_{b\sim [\Sigma(U_p^n)]_0}[q(b)-1]^2 = \EE\limits_{b\sim [\Sigma(U_p^n)]_0} \bigg[\bigg(\sum\limits_{k=1}^n \hat{q}(k) K_k(b) \bigg)^2\bigg]\]
Since the generalized Krawtchouk functions are orthogonal (as proven in Lemma 9.2), the non-diagonal products evaluate to 0. So, the square of the summation is just the summation of the squared terms that it contains. Thus, exploiting the orthogonality of the generalized Krawtchouk functions and the linearity of the expectations, we write:
    \begin{equation*}\begin{split}
    \EE\limits_{b\sim [\Sigma(U_p^n)]_0}[q(b)-1]^2 &= \EE\limits_{b\sim [\Sigma(U_p^n)]_0} \bigg[\sum\limits_{k=1}^n \hat{q}(k)^2 K_k(b)^2\bigg] \\
    &=  \sum\limits_{k=1}^n \hat{q}(k)^2 \EE\limits_{b\sim [\Sigma(U_p^n)]_0} [K_k(b)^2] \\
    &=  \sum\limits_{k=1}^n \frac{\EE[K_k(|s|_0)]^2}{{n\choose k}^2 (p-1)^{2n-2k}} \cdot \EE\limits_{b\sim [\Sigma(U_p^n)]_0} [K_k(b)^2]     \end{split}\end{equation*}
    Finally, we use Lemma 9.2 to write $\EE\limits_{b\sim [\Sigma(U_p^n)]_0} [K_k(b)^2]$ as $\langle K_k, K_k \rangle = {n\choose k} (p-1)^{n-k}$.
    \begin{equation*}\begin{split}
    \EE\limits_{b\sim [\Sigma(U_p^n)]_0} [q(b)-1]^2 &=  \sum\limits_{k=1}^n \frac{\EE[K_k(|s|_0)]^2}{{n\choose k}^2 (p-1)^{2n-2k}} {n\choose k} (p-1)^{n-k} \\
    &=  \sum\limits_{k=1}^n \frac{\EE[K_k(|s|_0)]^2}{{n\choose k} (p-1)^{n-k}} \qedhere
    \end{split}\end{equation*}
\end{proof}
\begin{theorem} For $\lambda\leq\frac{1}{1+e}$, \[\mathrm{TVD}([\Sigma(S(n,p,\lambda))]_0, [\Sigma(\mathrm U_p^n)]_0) \leq \sqrt{\EE\limits_{b\sim M}[p(b)-1]^2} \leq O(\lambda p^p)\] \end{theorem}
\begin{proof}
Substituting the result of Corollary 5.1 into the equation derived in Lemma 6.2, and scaling the indexes of the summation, we have that:
\begin{align*}
    \EE\limits_{b\sim [\Sigma(U_p^n)]_0} [q(b)-1]^2 &= \sum\limits_{k=1}^{n/p} \frac{1}{{n\choose pk}(p\!-\!1)^{n-pk}}\bigg((p\!-\!1)^{n-pk} \sum\limits_{d=pk}^{n-pk} {{d-1}\choose {\floor{\frac{kp}{p-1}}}-1}{{n-d}\choose{\floor{\frac{kp}{p-1}}}} \lambda^d\bigg)^2 \\
    &= \sum\limits_{k=1}^{n/p} \frac{(p\!-\!1)^{n-pk}}{{n\choose pk}}\bigg( \sum\limits_{d=pk}^{n-pk} {{d-1}\choose {\floor{\frac{k}{1-\frac{1}{p}}}}-1}{{n-d}\choose{\floor{\frac{k}{1-\frac{1}{p}}}}} \lambda^d\bigg)^2 \\
    &\leq \sum\limits_{k=1}^{n/p} (p\!-\!1)^{n-pk}\frac{{n\choose k}^2}{{n\choose pk}}\bigg( \sum\limits_{d=pk}^{n-pk} {{d-1}\choose {k-1}}\lambda^d\bigg)^2 
\end{align*}
Note the following generating function relation that $(\frac{x}{1-x})^k = \sum\limits_{m \geq k} {{m-1}\choose{k-1}}x^m$. Then,
\begin{align*}
     \EE\limits_{b\sim [\Sigma(U_p^n)]_0} [q(b)-1]^2 &\leq \sum\limits_{k=1}^{n/p} (p-1)^{n-pk}\frac{{n\choose k}^2}{{n\choose pk}}\bigg(\frac{\lambda}{1-\lambda}\bigg)^{2k}\\
    &\leq \sum\limits_{k=1}^{n/p}  (p\!-\!1)^{n-pk} \bigg(\frac{pk}{n}\bigg)^{pk}\bigg(\frac{en}{k}\bigg)^{2k} \bigg(\frac{\lambda}{1-\lambda}\bigg)^{2k} \quad \text{(From claim 6.1)}\\
&= \sum\limits_{k=1}^{n/p} (p\!-\!1)^{n-pk} \bigg(\frac{pk}{n}\bigg)^{pk-2k} \bigg(\frac{p e \lambda}{1-\lambda}\bigg)^{2k} \\
&\leq p^{2p} \sum\limits_{k=1}^{n/p} \bigg(\frac{e \lambda}{1-\lambda}\bigg)^{2k} \\
&\leq p^{2p} O(\lambda^2), \quad \text{for } \lambda \leq \frac{1}{1+e}
\end{align*}
Therefore, for $\lambda\leq\frac{1}{1+e}$, we have that $\mathrm{TVD}([\Sigma(S(n,p,\lambda))]_0, [\Sigma(\mathrm U_p^n)]_0) \leq \sqrt{\EE\limits_{b\sim M}[p(b)-1]^2} \leq O(\lambda p^{O(p)})$. \qedhere
\end{proof}

This method shows that the total variation distance between the $n$-step generalized sticky random walk on $p$ vertices and the $n$-ary samples from the Uniform distribution on $\ZZ_p$ is $O(\lambda p^{O(p)})$, matching the more general result in \cite{golowich_et_al:LIPIcs.CCC.2022.27}. 

\newpage

\end{document}